\documentclass[11pt, a4paper]{amsart}

\usepackage[dvipsnames]{xcolor}
\usepackage[pagebackref,backref=true,colorlinks=true, linkcolor=Sepia, citecolor=Sepia]{hyperref}
\usepackage{amsmath,amsthm,amssymb,fancyhdr,graphicx,bbm,cancel,mathrsfs,todonotes,xypic,pinlabel}
\usepackage{tikz}
\usepackage{extarrows,tipa}
\usetikzlibrary{matrix}
\usepackage[all]{xy}
\usepackage{mathtools}
\usepackage{bm}
\usepackage{verbatim}
\usepackage{microtype}
\usepackage{subcaption}

\DeclareFontFamily{T1}{cbgreek}{}
\DeclareFontShape{T1}{cbgreek}{m}{n}{<-6>  grmn0500 <6-7> grmn0600 <7-8> grmn0700 <8-9> grmn0800 <9-10> grmn0900 <10-12> grmn1000 <12-17> grmn1200 <17-> grmn1728}{}
\DeclareSymbolFont{quadratics}{T1}{cbgreek}{m}{n}
\DeclareMathSymbol{\qoppa}{\mathord}{quadratics}{19}
\DeclareMathSymbol{\Qoppa}{\mathord}{quadratics}{21}

\theoremstyle{theorem}%The commands below will have bold title, italic text
\newtheorem{thm}{Theorem}[section]
\newtheorem{mainthm}{Theorem}

\newtheorem{lem}[thm]{Lemma}
\newtheorem{prop}[thm]{Proposition}
\newtheorem{cor}[thm]{Corollary}
\newtheorem{add}[thm]{Addendum}
\theoremstyle{definition} %The commands below have bold title, standard text

\newtheorem{example}[thm]{Example}
 
\newtheorem{qu}[thm]{Question} 
\newtheorem{prob}[thm]{Problem} 
 
\newtheorem{rmk}[thm]{Remark}

\theoremstyle{remark} %The commands below have italic title, standard text
\usepackage{enumerate,pdfpages,stmaryrd} 
\usepackage[margin=1.2in, marginparwidth=1in]{geometry}

\usepackage{tikz}
\usepackage{dsfont}
\usetikzlibrary{matrix}

\newcommand{\C}{\mathbb{C}}

\newcommand{\Z}{\mathbb{Z}}

\newcommand{\R}{\mathbb{R}}

\newcommand{\beq}{\begin{equation*}}
\newcommand{\eeq}{\end{equation*}}

\newcommand{\ov}{\overline}
\newcommand{\wtil}{\widetilde}
\newcommand{\what}{\widehat}

\newcommand{\pair}[1]{\langle #1 \rangle}
\newcommand{\onto}{\twoheadrightarrow}

\newcommand{\SL}{\mathrm{SL}}

\DeclareMathOperator{\SO}{SO}

\DeclareMathOperator{\Diff}{Diff}

\DeclareMathOperator{\im}{im}

\DeclareMathOperator{\Aut}{Aut}

\DeclareMathOperator*{\colim}{colim}

\DeclareMathOperator{\id}{id}

\DeclareMathOperator{\Isom}{Isom}
\DeclareMathOperator{\Homeo}{Homeo}
\DeclareMathOperator{\TopO}{Top/O}
\DeclareMathOperator{\PLO}{PL/O}
\DeclareMathOperator{\OO}{O}
\DeclareMathOperator{\TopPL}{Top/PL}
\DeclareMathOperator{\Top}{Top}
\DeclareMathOperator{\Out}{Out}

\DeclareMathOperator{\Int}{int}

\newcommand{\rest}[2]{#1\bigr\vert_{#2}}

\date{\today}
\begin{document}

\title{Symmetries of exotic aspherical space forms}

\author{Mauricio Bustamante and Bena Tshishiku}

\maketitle

\begin{abstract}
We study finite group actions on smooth manifolds of the form $M\#\Sigma$, where $\Sigma$ is an exotic $n$-sphere and $M$ is a closed aspherical space form. We give a classification result for free actions of finite groups on $M\#\Sigma$ when $M$ is 7-dimensional. We show that if $\Z/p\Z$ acts freely on $T^n\#\Sigma$, then $\Sigma$ is divisible by $p$ in the group of homotopy spheres. When $M$ is hyperbolic, we give examples $M\#\Sigma$ that admit no nontrivial smooth action of a finite group, even though $\Isom(M)$ is arbitrarily large. Our proofs combine geometric and topological rigidity results with smoothing theory and computations with the Atiyah--Hirzebruch spectral sequence. 
\end{abstract}

\section{Introduction}

In this paper we study how smooth Lie group actions on a manifold $M$ depend on the smooth structure. The manifolds we consider are space forms, i.e.\ manifolds that admit a complete Riemannian metric with constant sectional curvature. 

The most classical case is when $M$ is the $n$-sphere $S^n$. A smooth manifold $\Sigma$ that is homeomorphic but not diffeomorphic to $S^n$ is called an exotic sphere. Several structural results about Lie group actions on exotic spheres are known by the work of Hsiang--Hsiang \cite{hsiang,hsiang-hsiang}. For example, the standard smooth structure on $S^n$ is characterized uniquely by its faithful action of $\SO(n+1)$. Furthermore, exotic spheres $\Sigma$ that bound parallelizable manifolds generally have more symmetry than those that do not; compare \cite[Main Thm.]{hsiang} and \cite[Thm.\ 1]{hsiang-hsiang}. 

This paper focuses on closed aspherical space forms $M$ (flat and hyperbolic manifolds), smooth structures $M\#\Sigma$ obtained by connected sum with an exotic sphere $\Sigma$, and finite group actions. This case is largely unexplored. The only prior result is the following theorem of Farrell--Jones \cite{FJ-nielsen}. In the statement, $\Theta_n$ denotes the group of homotopy $n$-spheres. When $n\ge 5$, every nontrivial element in $\Theta_n$ can be represented by an exotic $n$-sphere \cite{kervaire-milnor}.

\begin{thm}[Farrell--Jones]\label{thm:FJ}
Fix $n$ and assume $\Theta_n$ contains an element of order $\ge 3$. Then there exists a hyperbolic $n$-manifold $M$ and an exotic $n$-sphere $\Sigma$ such that $M$ has an orientation-reversing involution, but $M\#\Sigma$ has no orientation-reversing involution. 
\end{thm}

Our main results (Theorems \ref{thm:dim7}, \ref{thm:tori}, \ref{thm:asymmetric}) identify new phenomena and prove structural results in basic cases.

\subsection{Standard actions and the divisibility property} 
For a finite group $G$ acting on $M$ there is an obvious way to build an action on $M\#\Sigma$ for certain $\Sigma$. Specifically, for any exotic sphere $\what\Sigma$, there is an action of $G$ on 
\[M\#\underbrace{\what\Sigma\#\cdots\#\what\Sigma}_{|G|\text{ copies }}\] obtained by performing the connected sum equivariantly along the $G$-orbit of a point in $M$ with trivial stabilizer. We call an action $G\curvearrowright M\#\Sigma$ \emph{standard} if it is equivalent to this construction (see \S\ref{sec:standard} for more details on this notion). 

Our first result gives a class of manifolds and groups for which this is the only construction. 

\begin{mainthm}[Action classification, dimension 7]\label{thm:dim7} 
Let $M$ be a $7$-dimensional closed aspherical space form and let $\Sigma$ be an exotic $7$-sphere. Assume that $G$ is a finite group with odd order. 
Then every free, orientation-preserving action of $G$ on $M\#\Sigma$ is standard. 
\end{mainthm}

If $M$ is hyperbolic and $\Isom(M)$ acts freely on $M$, then every finite $G<\Homeo(M)$ acts freely on $M$ by Cappell--Lubotzky--Weinberger \cite[Thm.\ 1.5]{CLW}. For such $M$, the statement of Theorem \ref{thm:dim7} can be strengthened by dropping the word ``free". Every finite group $G$ acts freely by orientation-preserving isometries on some hyperbolic manifold by Belolipetsky--Lubotzky \cite{belolipetsky-lubotzky}, so there are plenty of examples to which Theorem \ref{thm:dim7} applies. 

As a consequence of Theorem \ref{thm:dim7}, if $G$ acts on $M\#\Sigma$ (with the assumptions of the theorem), then $\Sigma$ is divisible by $|G|$ in the group $\Theta_n$ of homotopy $n$-spheres (this implication uses the fact that inertia group vanishes for aspherical space forms -- see Proposition \ref{prop:inertia} and the paragraph after the proof). More generally, we say a group action of $G$ on $M\#\Sigma$ has the \emph{divisibility property} if $\Sigma$ is divisible by $|G|$ in $\Theta_n$ (i.e.\ there exists $\what\Sigma$ so that $\Sigma=\what\Sigma^{\#|G|}$ in $\Theta_n$). This is generally weaker than the action being standard; see Example \ref{ex:divisibility}. 

Without the assumption that $|G|$ is odd in Theorem \ref{thm:dim7}, we can prove the following.
\begin{add}\label{add:hyperbolic7}
Let $M$ and $\Sigma$ be as in the statement of Theorem \ref{thm:dim7}. Then every free, orientation-preserving action of a finite group $G$ on $M\#\Sigma$ has the divisibility property.
\end{add}
For actions on exotic tori we obtain a similar result but with no constraint on the dimension. 
\begin{mainthm}[Actions on exotic $n$-tori]\label{thm:tori}
Fix $n\ge5$ and let $T^n=(S^1)^n$ denote the $n$-torus. Let $\Sigma$ be an exotic $n$-sphere. Assume that $G$ is a cyclic group of prime order. Then any free, orientation-preserving action of $G$ on $T^n\#\Sigma$ has the divisibility property. 
\end{mainthm}

\subsection{Symmetry constant and asymmetric manifolds} 

For an aspherical manifold $N$, the \emph{symmetry constant} $s(N)$ is defined as the maximum size of a finite subgroup of $\Diff(N)$. 

By a result of Borel \cite{borel-isometry}, if $\pi_1(N)$ has trivial center, then $s(N)\le |\Out(\pi_1(N))|$. When $M$ is a hyperbolic manifold of dimension $\ge3$, then 
\[s(M)=|\Isom(M)|=|\Out(\pi_1(M))|.\]
The second equality follows from Mostow rigidity. 
Theorem \ref{thm:FJ} of Farrell--Jones \cite{FJ-nielsen} gives examples of hyperbolic manifolds $M$ and exotic spheres $\Sigma$ such that 
\[s(M\#\Sigma)\le \frac{1}{2}|\Isom(M)|.\] This leads one to wonder how small $s(M\#\Sigma)$ can be relative to the size of $|\Isom(M)|$. The most extreme possibility would be a positive answer the following question. 

\begin{qu}\label{q:asymmetric}
Fix $n$. Given $d\ge1$, does there exist a hyperbolic $n$-manifold $M$ and an exotic $n$-sphere $\Sigma$ such that $|\Isom(M)|\ge d$ and $s(M\#\Sigma)=1$? 
\end{qu}

The condition $s(N)=1$ is equivalent to the statement that $N$ is  \emph{(smoothly) asymmetric}, i.e.\ $\Diff(N)$ does not contain a nontrivial finite subgroup. Theorem \ref{thm:asymmetric} provides an answer to Question \ref{q:asymmetric} and also answers Question 3 of \cite{BT-hyperbolic} (note that the definition of $s(M\#\Sigma)$ used in \cite{BT-hyperbolic} differs from the one above by a factor of $|\Isom(M)|$). 

\begin{mainthm}[Asymmetric smoothings of hyperbolic manifolds]\label{thm:asymmetric}
For every $n_0\geq 5$ and $d\ge1$, there exists $n\ge n_0$, a closed hyperbolic $n$-manifold $M$, and an exotic sphere $\Sigma\in\Theta_n$, so that $|\Isom(M)|\ge d$ and $M\#\Sigma$ is (smoothly) asymmetric. 
\end{mainthm}

The manifolds in Theorem \ref{thm:asymmetric} have many topological symmetries (because $\Isom(M)$ is large and $N$ and $M$ are homeomorphic) but no smooth symmetries. This is a new phenomenon in the study of asymmetric manifolds; previous results prove asymmetry by controlling $\Out\big(\pi_1(M)\big)$. The first examples of asymmetric aspherical manifolds were constructed by Conner--Raymond--Weinberger \cite{CRW}. These examples (some of which are solvmanifolds) are \emph{topologically asymmetric} (i.e.\ $\Homeo(M)$ does not contain any nontrivial finite subgroup), and  they are shown to be asymmetric by arranging that $\Out(\pi_1(M))$ is torsionfree. In the Riemannian category, Long--Reid \cite{LR-asymmetric} gave examples of hyperbolic $n$-manifolds ($n\ge2$) that are \emph{isometrically asymmetric} (i.e.\ $\Isom(M)\cong\Out(\pi_1(M))=1$). See also \cite{belolipetsky-lubotzky}. 

The simplest nontrivial instance of Theorem \ref{thm:asymmetric} is as follows. Let $M$ be a hyperbolic 7-manifold such that $\Isom(M)=\Isom^+(M)=\Z/7\Z$ acts freely on $M$ (such examples exist by \cite{belolipetsky-lubotzky}), %c.f.\ \cite[Thm.\ 6]{BT-hyperbolic}). 
and let $\Sigma$ be a generator of $\Theta_7\cong\Z/28\Z$. Then $M\#\Sigma$ is asymmetric. We could also deduce that this example is asymmetric using Theorem \ref{thm:dim7}.

\subsection{About the proofs} For the following discussion, recall from smoothing theory that concordance classes of smooth structures on a smooth manifold $M$ are in bijection with homotopy classes of maps $M\to\TopO$. This is discussed more in \S\ref{sec:smoothing}.

{\it About Theorem \ref{thm:dim7}.} To show the action $G\curvearrowright M\#\Sigma$ is standard, it suffices to show that the quotient $(M\#\Sigma)/G$ is diffeomorphic to a manifold of the form $(M/G)\#\what\Sigma$ for some smooth action $G\curvearrowright M$ that is topologically conjugate to the action $G\curvearrowright M\#\Sigma$ (see Lemma \ref{lem:standard-actions}). Using rigidity results for hyperbolic and flat manifolds (Borel conjecture, Mostow and Bieberbach rigidity) we reduce the problem to computing a homomorphism 
\[\pi^*:[M/G,\TopO]\to[M,\TopO]\]
induced by a covering map $M\to M/G$. To compute $\pi^*$, we use results about the topology of $\TopO$, including the fact that it is an infinite loop space, which enables the use of the Atiyah--Hirzebruch spectral sequence.

{\it About Theorem \ref{thm:tori}.} To show the action $G\curvearrowright T^n\#\Sigma$ has the divisibility property, we use a similar strategy to Theorem \ref{thm:dim7}. Again, we want to compute a homomorphism 
\[\pi^*:[T^n/G,\TopO]\to[T^n,\TopO].\]
Specifically, we want to know the pre-image of $T^n\#\Sigma$ under $\pi^*$. From the point-of-view of the Atiyah--Hirzebruch spectral sequence, the divisibility property is related to the filtration 
\[F_0\subset F_1\subset\cdots\subset F_r=[M,\TopO]\]
whose associated graded appears on the $E_\infty$ page, and whether or not any of the extensions 
\[0\rightarrow F_0\rightarrow F_k\rightarrow F_k/F_0\rightarrow 0\]
are split. Here $F_0\cong\Theta_n/I(M)$, where $I(M)=\{\Sigma\in\Theta_n: M\#\Sigma\cong M\}$ is the inertia group of $M$. We are unable to resolve the necessary extension problems in general. Instead, we use some structural results about flat manifolds (Propositions \ref{prop:mapping-torus} and \ref{prop:parallelizable}) and a homotopical property of the (suspension of the) quotient map $T^n\to T^n/G$ that we call ``compatible splitting of the top cell"; see Proposition \ref{prop:compatible-splitting}. This is proved by carefully constructing an equivariant Whitney embedding of $T^n$ with its $G$ action.

{\it About Theorem \ref{thm:asymmetric}.} 
Recall that Farrell--Jones prove Theorem \ref{thm:FJ} above by finding $M,\Sigma$ with the property that the natural homomorphism 
\[\Diff(M\#\Sigma)\rightarrow\Out(\pi_1(M\#\Sigma))\] is not surjective because its image is contained in $\Out^+\big(\pi_1(M\#\Sigma)\big)$. It is impossible to improve their result with this approach because the image of this homomorphism always contains $\Out^+\big(\pi_1(M\#\Sigma)\big)$; c.f.\ \cite[Thm.\ 1]{BT-hyperbolic}. 

For Theorem \ref{thm:asymmetric} we use a result of Belolipetsky--Lubotzky \cite{belolipetsky-lubotzky} that every finite group is the orientation-preserving isometry group of some hyperbolic $n$-manifold. With this fact, it would be easy to prove the theorem if we knew that every cyclic action $G\curvearrowright M\#\Sigma$ satisfies the divisibility property, for then we could choose $M,\Sigma$ with the property that $\Sigma$ is not divisible by the order of any nontrivial element of $\Isom^+(M)$. Unfortunately, we don't know an analogue of Theorem \ref{thm:tori} for high-dimensional hyperbolic manifolds. Instead we use again the Atiyah--Hirzebruch spectral sequence, but now for $[M,\TopO_{(p)}]$ where $\TopO_{(p)}$ is the localization at a prime $p$. 

\subsection{Further direction} \

One of the motivations that led to this work is the following question related to the Zimmer program (see \cite{fisher1, fisher2}). 

\begin{qu}\label{q:zimmer}
Let $T'$ be an exotic $n$-torus. Does there exist a smooth, faithful action of $\SL_n(\Z)$ on $T'$? 
\end{qu}

There seems to be no obvious action of $\SL_n(\Z)$ on any exotic torus $T'$, so one might suspect that the answer to Question \ref{q:zimmer} is ``no". As a step toward proving this, we suggest the following problem, which is an example of a Nielsen realization problem. 

\begin{prob}\label{prob:torus-realization}
For an exotic $n$-torus $T'$, show that the natural homomorphism 
\[\Diff^+(T')\rightarrow\SL_n(\Z)\] is not a split surjection. 
\end{prob}

The only interesting case of Problem \ref{prob:torus-realization} is $T'=T^n\#\Sigma$ when $\Sigma\neq S^n\in\Theta_n$. This is because, for any other smooth exotic structure $T'$, the homomorphism $\Diff^+(T')\to\SL_n(\Z)$ is not surjective. This can be proved similar to \cite[Thm.\ 1]{BT-hyperbolic}. 

Theorem \ref{thm:tori} provides some partial progress toward Problem \ref{prob:torus-realization}, as we now explain. For $G<\SL_n(\Z)$, we say that $G$ is \emph{realized by diffeomorphisms} of $T^n\#\Sigma$ if there exists a solution to the following lifting problem
\begin{equation}\label{eqn:nielsen-tori}\begin{xy}
(0,0)*+{\Diff^+(T^n\#\Sigma)}="A";
(-20,-15)*+{G}="B";
(0,-15)*+{\SL_n(\Z)}="C";
(13,-7.5)*+{\text{action on $\pi_1$}}="D";
{\ar@{-->}"B";"A"}?*!/_3mm/{};
{\ar@{^{(}->} "B";"C"}?*!/_3mm/{};
{\ar "A";"C"}?*!/_3mm/{};
\end{xy}\end{equation}
 
If there exists a lift where $G$ acts freely on $T^n\#\Sigma$, then we say $G$ is realized \emph{freely}. Theorem \ref{thm:tori} implies the following corollary. 

\begin{cor}[Nielsen realization for $T^n\#\Sigma$]\label{cor:nielsen-tori} Fix a prime $p$ and a subgroup $G\cong\Z/p\Z$ of $\SL_n(\Z)$. If $G$ is realized freely on $T^n\#\Sigma$, then $\Sigma$ is divisible by $p$ in $\Theta_n$. 
\end{cor} 

\begin{example}\label{ex:torus}
Let $G\cong\Z/7\Z<\SL_7(\Z)$ be a subgroup that is induced by an affine, free action of $\Z/7\Z$ on $T^7$ (it is easy to construct examples like this). 

By Corollary \ref{cor:nielsen-tori} and the standard action construction, $G$ is realized freely on $T^7\#\Sigma$ if and only if $\Sigma$ belongs to the subgroup of $\Theta_7\cong\Z/28\Z$ that's isomorphic to $\Z/4\Z$. 
\end{example}

A possible approach to Problem \ref{prob:torus-realization} would be to prove an analogue of Theorem \ref{thm:tori} for non-free actions. We intend to pursue this in a future work. 

\subsection*{Section Outline.} In \S\ref{sec:standard} we give a characterization of standard actions and compare this notion with the divisibility property. In \S\ref{sec:smoothing} we explain the necessary background from smoothing theory. We prove Theorem \ref{thm:dim7} in \S\ref{sec:dim7}. In \S\ref{sec:flat} we prove two results about flat manifolds, which we use to prove Theorem \ref{thm:tori} in \S\ref{sec:tori}. Theorem \ref{thm:asymmetric} is proved in \S\ref{sec:asymmetric}.

\subsection*{Conventions/Notations} In the rest of the paper, all manifolds are closed, oriented, and connected; all actions preserve the orientation; $G$ is a finite group; and $\Sigma$ is an exotic $n$-sphere. Usually we use $M$ to denote an aspherical space form, and we use $W$ to denote a more general smooth manifold. We use $\mathrm{Homeo}(W)$ and $\Diff(W)$ to denote the homeomorphism and diffeomorphism groups of $W$ with the compact-open and compact-open-$C^{\infty}$ topologies respectively, and we write $\Homeo^+(W)$ and $\Diff^+(W)$ for the orientation-preserving subgroups.

\subsection*{Acknowledgements.} The second author thanks S.\ Cappell for telling him about \cite{CLW}, which inspired aspects of this project. The authors thank A.\ Kupers for sharing his result \cite{kupers} about the $k$-invariants of $\TopO$, which relates to Addendum \ref{add:hyperbolic7}. 

\section{Standard actions}\label{sec:standard}

In this section we carefully define a standard action and we give a characterization of standard actions (Lemma \ref{lem:standard-actions}) that we will use to prove Theorem \ref{thm:dim7}. 

\subsection{$G$-spaces}
We begin with some terminology. We work in the category of smooth manifolds with a faithful, orientation-preserving $G$ action. By a \emph{smooth $G$-space} we mean a pair $(W,\rho)$ where $W$ is a smooth manifold  and $\rho:G\rightarrow\Diff^+(W)$ is an injective homomorphism. Two smooth $G$-spaces are isomorphic if there is an orientation-preserving diffeomorphism $W\rightarrow W'$ that intertwines the $G$ actions. Any smooth $G$-space has an underlying topological $G$-space where we forget the smooth structure. 

We will always suppress $\rho$ from the notation and refer to $W$ has a $G$-space; we even write $G<\Diff^+(W)$. This should not cause any confusion because each manifold we consider will only ever have one $G$ action of interest. 

\subsection{Equivariant connected sum}

Let $W$ be any smooth $G$-space, and fix an embedded $n$-disk $D\subset W$ whose translates under $G$ are disjoint. For any manifold $U$, the equivariant connected sum is obtained by forming the connected sum of $W$ with a copy of $U$ along each of the disks $\{g(D): g\in G\}$. This makes $W\#U\#\cdots\#U$ ($|G|$ copies of $U$) into a smooth $G$-space.

\subsection{Standard actions} 
Fix a smooth $n$-manifold $W$. In what follows it's helpful to think of $W$ as the ``preferred" smooth structure on the underlying topological manifold.

Let $\Sigma$ be an exotic $n$-sphere, and assume that $W\#\Sigma$ is a smooth $G$-space. We say that $W\#\Sigma$ is \emph{standard} (as a $G$-space) if there is a smooth $G$-space structure on $W$ and an exotic sphere $\what\Sigma$ so that $W\#\Sigma$ is isomorphic to the equivariant connected sum $W\#\what\Sigma\#\cdots\#\what\Sigma$. 

The following lemma gives an alternate characterization of standard actions with the additional assumption that the action is free.

\begin{lem}[Standard actions]\label{lem:standard-actions}
Let $W$ be a closed smooth $n$-manifold, and fix $\Sigma\in\Theta_n$. Assume that $W\#\Sigma$ is a smooth $G$-space with $G$ acting freely. Then $W\#\Sigma$ is standard if and only if 
there is a smooth $G$-space structure on $W$ such that (i) $W\#\Sigma$ and $W$ are isomorphic as topological $G$-spaces, and (ii) the quotient manifold $(W\#\Sigma)/G$ is diffeomorphic to $(W/G)\#\what\Sigma$ for some $\what\Sigma\in\Theta_n$. 
\end{lem}

\begin{proof}[Proof of Lemma \ref{lem:standard-actions}]
For the forward direction: if $W\#\Sigma$ is standard, then by definition there is a smooth $G$-space structure on $W$, an exotic sphere $\what\Sigma$, and an equivariant diffeomorphism $F:W\#\Sigma\rightarrow W\#\what\Sigma\#\cdots\#\what\Sigma$ to the equivariant connected sum. This implies (i) because the equivariant connected sum is isomorphic to $W$ as a topological $G$-space (by equivariant coning/Alexander trick). Furthermore, since $F$ is an equivariant diffeomorphism and the $G$-actions are free, it descends to a diffeomorphism 
\[(W\#\Sigma)/G \rightarrow(W\#\what\Sigma\#\cdots\#\what\Sigma)/G\cong (W/G)\#\what\Sigma.\]

For the reverse direction, suppose given a $G$-space structure on $W$ and an exotic sphere $\what\Sigma$ so that properties (i) and (ii) hold, and write 
\[f:(W/G)\#\what\Sigma\rightarrow(W\#\Sigma)/G\]
for the diffeomorphism given by (ii).  We want to show there is an equivariant diffeomorphism between $W\#\Sigma$ and the equivariant connect sum $W\#\what\Sigma\#\cdots\#\what\Sigma$. 

On the one hand, the covering map $W\#\Sigma\rightarrow(W\#\Sigma)/G$ pulls back along $f$ to a covering 
\[p:X\rightarrow(W/G)\#\what\Sigma.\]  

On the other hand, the given action of $G$ on $W$ induces a standard action $G\curvearrowright W\#(\what\Sigma^{\#|G|})$ which defines a smooth covering space 
\[p':W\#(\what\Sigma^{\#|G|})\rightarrow (W/G)\#\what\Sigma.\]

The lemma follows by showing that $X$ and $W\#(\what\Sigma^{\#|G|})$ are (smoothly) isomorphic covering spaces of $(W/G)\#\what\Sigma$. Since the smooth structure on a smooth cover of a smooth manifold is uniquely determined, it suffices to observe that $p$ and $p'$ are isomorphic in the topological category. This holds because both are topologically equivalent to the covering $W\rightarrow W/G$. This holds for $p'$ by construction, and holds for $p$ by our assumption that $G\curvearrowright W\#\Sigma$ and $G\curvearrowright W$ are topologically conjugate. 

Put another way, we have the following diagram. 
\[\begin{xy}
(-70,0)*+{X}="A";
(-40,0)*+{W\#\Sigma}="B";
(-10,0)*+{W}="C";
(20,0)*+{W\#(\what\Sigma)^{\#|G|}}="D";
(-70,-15)*+{(W/G)\#\what\Sigma}="F";
(-40,-15)*+{(W\#\Sigma)/G}="G";
(-10,-15)*+{W/G}="H";
(20,-15)*+{(W/G)\#\what\Sigma}="I";
{\ar"A";"B"}?*!/_3mm/{\cong_{\Diff}};
{\ar"A";"F"}?*!/^3mm/{p};
{\ar "B";"C"}?*!/_3mm/{\cong_{\Top}};
{\ar "D";"C"}?*!/^3mm/{};
{\ar"F";"G"}?*!/_3mm/{f};
{\ar "G";"H"}?*!/_3mm/{};
{\ar "I";"H"}?*!/^3mm/{};
{\ar "B";"G"}?*!/_3mm/{};
{\ar "C";"H"}?*!/^3mm/{};
{\ar "D";"I"}?*!/_3mm/{p'};
\end{xy}\]
In order for the covering spaces $p,p'$ to be equivalent, we want each square in this diagram to commute. The left-most square commutes by construction. The middle square commutes by the assumption (i) that $W\#\Sigma$ and $W$ are isomorphic as topological $G$-spaces. In the right-most square the horizontal maps are given by coning on the bottom and equivariant coning on the top, so this square also commutes. 
\end{proof}

\subsection{Standard actions vs.\ divisibility property}

If $G\curvearrowright W\#\Sigma$ is standard, then there exists $\what \Sigma$ and an orientation-preserving diffeomorphism 
\[W\#\Sigma\cong W\#(\what\Sigma^{\#|G|}).\]
When $W$ is a stably parallelizable aspherical space form, this implies that $[\Sigma]=[\what\Sigma^{\#|G|}]$ in $\Theta_n$ by Proposition \ref{prop:inertia}. Hence a standard action on such manifolds has the divisibility property. The following Example \ref{ex:divisibility} demonstrates that having the divisibility property is weaker than being standard in general. 

\begin{example}\label{ex:divisibility}
Let $\Sigma\in\Theta_7$ be a nontrivial element of order 7. Let $G=\Z/7\Z$ act on $T^7$ by rotations and form an equivariant connected sum to get an action of $G$ on $T^7\#\Sigma^{\#7}$. Here $T^7\#\Sigma^{\#7}$ is diffeomorphic to $T^7$, but they are not isomorphic as $G$-spaces because the quotients $T^7/G\cong{T^7}$ and $(T^7\#\Sigma^{\#7})/G\cong T^7\#\Sigma$ are not diffeomorphic (see Proposition \ref{prop:inertia} below). Let now $\Omega\in\Theta_8\cong\Z/2\Z$ be the nontrivial element, and define a $G$-space 
\[N=\big[(T^7\#\Sigma^{\#7})\times S^1\big]\#\,\Omega^{\#7},\]
where $G$ acts trivially on the second factor of $(T^7\#\Sigma^{\#7})\times S^1$, and the connect sum with $\Omega^{\#7}$ is performed equivariantly. By construction $N$ is diffeomorphic to $T^8\#\Omega$, but the action of $G$ on $N$ is not standard by Lemma \ref{lem:standard-actions} because the quotient 
\[N/G\cong\big[(T^7\#\Sigma)\times S^1\big]\#\Omega\]
is not diffeomorphic to $T^8\#\,\Omega$ (if they are, then take connected sum with $\Omega$ to deduce that $T^8$ and $(T^7\#\Sigma)\times S^1$ are diffeomorphic, a contradiction). Note however that the divisibility property is satisfied since $\Omega=\Omega^{\#7}$. 
\end{example}

\begin{rmk}
It would be interesting to expand the notion of a standard action to include the action in Example \ref{ex:divisibility} and to prove a more general classification theorem.  
\end{rmk}

\begin{qu}
Does there exist a non-positively curved manifold $W$, an exotic sphere $\Sigma$, and a free, orientation-preserving action of a finite group $G$ on $W\#\Sigma$ that does not satisfy the divisibility property? 
\end{qu}

\section{Background from smoothing theory} \label{sec:smoothing}

In this section we collect various results from smoothing theory. For an informal treatment of the foundational results, we recommend the notes of Davis \cite{davis}; a more formal treatment is given in Siebenmann's \cite[Essay IV]{kirby-siebenmann}. In addition to the foundational results, we include a basic result about the concordance classes of smooth structures $W\#\Sigma$ (Lemma \ref{lem:concordance}), and a computation of the inertia group for closed aspherical space forms (Proposition \ref{prop:inertia}). 

\subsection{Smooth structures}

A \emph{smooth structure} or \emph{smoothing} of a topological manifold $W$ is a choice of maximal smooth atlas. A \emph{smooth marking} on $W$ is a pair $(U,h)$ where $U$ is a smooth manifold and $h:U\rightarrow W$ is a homeomorphism; this in particular determines a smoothing by pushing forward a smooth atlas. Two marking $(U,h)$ and $(U',h')$ determine the same smoothing if and only if there is a diffeomorphism $\phi:U'\rightarrow U$ so that $h'=h\circ\phi$. 

Two smoothings of $W$ are \emph{concordant} if they can be realized as the boundary of a smoothing of $W\times[0,1]$. A concordance can be represented as a marking $V\to W\times[0,1]$.   

We denote the set of concordance classes of smooth structures on $W$ by $S(W)$.

In the following classification theorem \cite[Essay IV, Thm.\ 10.1]{kirby-siebenmann}, $\TopO$ denotes the homotopy fiber of the natural map $B\OO\rightarrow B\Top$, where $\OO=\colim \OO(n)$ is the infinite orthogonal group, and $\Top=\colim \Homeo(\R^n,0)$, with $\Homeo(\R^n,0)$ the topological group of homeomorphisms of $\R^n$ that fix the origin. We denote by $[W,\TopO]$ the set of homotopy classes of based maps; this is equivalent to the set of homotopy classes of unbased maps since $\TopO$ is an $H$-space under Whitney sum (and because we assume $W$ is connected).
\begin{thm}Let $W$ be a closed topological manifold of
dimension $n\ge5$. Then a choice of a smooth structure $U\xrightarrow{h}W$ on $W$ determines a bijection $S(W)\xrightarrow{\cong} [W,\TopO]$ under which $(U,h)$ is sent to the homotopy class of the constant map.
\end{thm}

To compute $[W,\TopO]$, we use that $\TopO$ is in fact an infinite loop space \cite[p.216]{Boardman-Vogt}. This affords two different approaches: 
\begin{enumerate}
\item For any $k\ge1$, if we write $\TopO=\Omega^kY$, then \[[W,\TopO]\cong[W,\Omega^kY]\cong[\Sigma^kW,Y].\]
Thus information about the homotopy type of $\Sigma^kW$ (which can be simpler than that of $W$) can allow us to make conclusions about $[W,\TopO]$. 
\item We can view $[W,\TopO]$ as the $0$-th group of a cohomology theory. Then the Atiyah--Hirzebruch spectral sequence can be used to gain information about $[W,\TopO]$. We elaborate on this more below. 
\end{enumerate} 

\subsection{Smooth structures of the form $W\#\Sigma$.} 
View $W\#\Sigma$ as $(W\setminus \Int D^n)\cup_\phi D^n$, where $D^n$ is glued to $W\setminus\Int D^n$ along the common boundary $\partial D^n= S^{n-1}$ by $\phi\in\Diff^+(S^{n-1})$ whose isotopy class $[\phi]\in\pi_0(\Diff^+(S^{n-1}))\cong\Theta_n$ corresponds to $[\Sigma]$. From this point-of-view, there is a ``standard'' homeomorphism $\iota:W\#\Sigma\rightarrow W$, which is the identity on $W\setminus\Int D^n$ and $\rest{\iota}{D^n}$ is the cone of $\phi$ (i.e.\ choose an isotopy from $\phi$ to the identity in the connected group $\Homeo^+(S^{n-1})$). 

The map $W\rightarrow\TopO$ that classifies $(W\#\Sigma,\iota)$ can be obtained by a composition 
\[W\xrightarrow{c} S^n\xrightarrow{f}\TopO,\] where $c$ collapses the complement of $D^n\subset W$ to a point, and $f$ classifies $\Sigma\in S(S^n)\cong\pi_n(\TopO)$. 

In the following lemma, we examine the action of $\Homeo(W)$ on $S(W)$ given by post-composition. 

\begin{lem}\label{lem:concordance}
Fix $n\ge5$. Let $W$ be a smooth $n$-manifold. Fix $\Sigma\in\Theta_n$ and let $\iota:W\#\Sigma\rightarrow W$ be a coning homeomorphism. If every homeomorphism of $W$ is pseudo-isotopic to a diffeomorphism, then the markings $(W\#\Sigma,\iota)$ and $(W\#\Sigma,h\circ\iota)$ determine concordant smoothings for each $h\in\Homeo^+(W)$. 
\end{lem}

\begin{proof}[Proof of Lemma \ref{lem:concordance}]
The proof is straightforward. We give the details for completeness. Let $g$ be a diffeomorphism of $W$ that's pseudo-isotopic to $h$. Then the smoothings given by the markings $(W\#\Sigma,h\circ\iota)$ and $(W\#\Sigma,g\circ\iota)$ are concordant since homeomorphisms of $W$ that are pseudo-isotopic to the identity act trivially on $S(W)$. 

Now we show directly that $(W\#\Sigma,g\circ\iota)$ is concordant to $(W\#\Sigma,\iota)$. 
Let $D\subset W$ be the $n$-disk where the connected sum $W\#\Sigma$ is performed. Since $g$ preserves orientation, there is  an isotopy of $g_t$ with $g_0=g$ and such that $g_1$ is the identity on a neighborhood of $D$. The homeomorphism $(W\#\Sigma)\times[0,1]\rightarrow W\times[0,1]$ defined by $(x,t)\mapsto (g_t\circ\iota(x),t)$ gives a concordance between $(W\#\Sigma,g\circ\iota)$ and $(W\#\Sigma,g_1\circ\iota)$, but $g_1\circ\iota=\iota\circ g_1'$ for a diffeomorphism of $W\#\Sigma$ since $g_1$ is the identity near $D$. Thus the markings $(W\#\Sigma,g_1\circ\iota)$ and $(W\#\Sigma,\iota)$ determine the same smoothing, so the smoothings $(W\#\Sigma,g\circ\iota)$ and $(W\#\Sigma,\iota)$ are concordant.
\end{proof}

\begin{rmk}
Another way to state the conclusion of Lemma \ref{lem:concordance} is that the concordance class of $(W\#\Sigma,f)$ is independent of the choice of homeomorphism $f:W\#\Sigma\to W$ (since $f=(f\circ\iota^{-1})\circ\iota$). As such, we can safely omit the marking when referring to the smoothing $W\#\Sigma$. 
\end{rmk}

\begin{rmk}[Applying Lemma \ref{lem:concordance}]\label{rmk:concordance}
If $M$ is a closed aspherical space form of dimension $\ge5$, then $M$ satisfies the assumption of Lemma \ref{lem:concordance}. By the Borel conjecture \cite{farrell-hsiang,FJ-borel-hyperbolic}, the quotient of $\Homeo(M)$ by the group of homeomorphism pseudo-isotopic to the identity is $\Out\big(\pi_1(M)\big)$; see \cite[proof of Lem.\ 7]{davis}. When $M$ is hyperbolic, this group is $\Isom(M)$ by Mostow rigidity. When $M$ is flat, every automorphism of $\pi_1(M)$ is induced by an affine diffeomorphism by Bieberbach's theorem \cite[Ch.\ I, Thm.\ 4.1]{charlap}. 
\end{rmk}

\subsection{Inertia group for aspherical space forms}

Recall that for a smooth manifold $W$, the inertia group $I(W)$ is the subgroup of $\Theta_n$ consisting of all homotopy spheres $\Sigma\in\Theta_n$ such that $W\#\Sigma$ is diffeomorphic to $W$ by an orientation-preserving diffeomorphism. In this section we show that if $W$ is a stably parallelizable aspherical space form, then $I(W)$ is trivial. In particular, this shows that there are plenty of exotic aspherical space forms to which our results apply. 

A smooth manifold is \textit{stably parallelizable} if it can be embedded with trivial normal bundle in some Euclidean space. We say that a closed aspherical manifold \textit{satisfies the Borel conjecture} if every self-homotopy equivalence is homotopic to a homemorphism.
\begin{prop}\label{prop:inertia}
Let $W$ be a closed aspherical manifold. Assume
\begin{enumerate}[(i)]
\item $W$ is stably parallelizable; 
\item $W$ satisfies the Borel conjecture; 
\item the natural map $\Diff(W)\to\Out\big(\pi_1(W)\big)$ is surjective. 
\end{enumerate} 
Then the inertia group $I(W)$ is trivial. 
\end{prop}

Proposition \ref{prop:inertia} can be applied to a closed aspherical space form after passing to a finite cover. Regarding condition $(i)$, every closed hyperbolic manifold has a finite sheeted cover that is stably parallelizable \cite{sullivan, okun}. The same statement holds for flat manifolds since every flat manifold is finitely covered by $T^n$ by Bieberbach's (first) theorem \cite[Ch.\ I, Thm.\ 3.1]{charlap}. Furthermore, there are lots of flat manifolds different from $T^n$ that are parallelizable; see \cite{thorpe}. Conditions $(ii)-(iii)$ hold for all aspherical space forms of dimension $\ge5$. The Borel conjecture was proved in the flat case by Farrell--Hsiang \cite{farrell-hsiang} and in the hyperbolic case by Farrell--Jones \cite{FJ-borel-hyperbolic}. Condition $(iii)$ holds for hyperbolic manifolds by Mostow rigidity and holds for flat manifolds by Bieberbach's (second) theorem \cite[Ch.\ I, Thm.\ 4.1]{charlap}. 

A proof of Proposition \ref{prop:inertia} is given in the hyperbolic case in \cite[\S2]{FJ-exotic-neg}. The general case is similar.

\begin{proof}[Proof of Proposition \ref{prop:inertia}]
Fix $\Sigma\in\Theta_n$. Suppose that there is an orientation-preserving diffeomorphism $W\cong W\#\Sigma$. We want to show $[\Sigma]=[S^n]$ in $\Theta_n$. 

First we show that if there exists a diffeomorphism $f:W\#\Sigma\to W$, then the markings $(W,\id)$ and $(W\#\Sigma,\iota)$ give concordant smooth structures. Here $\iota$ is the coning map. By assumption, there exists a diffeomorphism $g:W\to W$ so that the homeomorphism $\iota\circ f^{-1}\circ g^{-1}:W\to W$ is homotopic to the identity. Since $g\circ f$ is a diffeomorphism, the markings $(W\#\Sigma,\iota)$ and $(W,\iota\circ f^{-1}\circ g^{-1})$ determine the same smoothing. We will show that $(W,\id)$ and $(W,\iota\circ f^{-1}\circ g^{-1})$ are concordant. A homotopy from $\iota\circ f^{-1}\circ g^{-1}$ to the identity gives a map $h:W\times [0,1]\to W\times[0,1]$ that restricts to a homeomorphism on the boundary. Since we assume the Borel conjecture for $W$, the map $h$ is homotopic rel boundary to a homeomorphism $h'$. Giving the domain of $h'$ the product smooth structure, the  homeomorphism $h'$ gives a smooth marking that is a concordance between $(W,\id)$ and $(W,\iota\circ f^{-1}\circ g^{-1})$, as desired.  

It remains to show that if $W$ and $W\#\Sigma$ are concordant, then $[\Sigma]=[S^n]$ in $\Theta_n$. Let $D^n\hookrightarrow W$ be an embedded disk, and let $c:W\to S^n$ be the map that collapses the complement of $D^n$ to a point. This induces a map 
\[\Theta_n\cong[S^n,\TopO]\xrightarrow{c^*}[W,\TopO]\cong S(W)\]
It suffices to show that $c^*$ is injective. Since $\TopO$ is an infinite loop space, if we write $\TopO=\Omega^kY$, then $[-,\TopO]\cong[\Sigma^k(-),Y]$ (adjunction), so it is equivalent to show that the map 
\begin{equation}\label{eqn:delooped}[S^{n+k},Y]\to[\Sigma^kW,Y]\end{equation}
is injective for some $k$. Since $W$ is stably parallelizable, for $k\ge n$, there is a homotopy equivalence $S^{n+k}\vee Z_W\to\Sigma^kW$, where $Z_W$ is a finite CW complex of dimension $<n+k$; see Lemma \ref{lem:splitting}. This implies that the map (\ref{eqn:delooped}) is a (split) injection.  
\end{proof}
It follows from \cite{kupers} that in dimension $7$ the map $c^*:\Theta_7\to [M,\TopO]$ is injective for all closed $7$-manifolds $M$. Hence the inertia group of \textit{any} $7$-dimensional closed aspherical space form is trivial.

\begin{rmk}
The argument above shows that if $W\#\Sigma_1$ is diffeomorphic to $W\#\Sigma_2$, then $[\Sigma_1]=\pm[\Sigma_2]$ in $\Theta_n$. The sign depends on whether the diffeomorphism preserves or reverses orientation. (If $f:W\#\Sigma_1\to W\#\Sigma_2$ reverses orientation, then by taking connected sum with $\Sigma_2$ in the domain, we obtain a diffeomorphism $W\#\Sigma_1\#\Sigma_2\to W\#\Sigma_2\#\overline{\Sigma}_2\cong W$.) 
\end{rmk}

\subsection{Localization of $\TopO$.} For the proof of Theorem \ref{thm:asymmetric}, we will use the localization $\TopO_{(p)}$ of $\TopO$ at an odd prime $p$. The homotopy groups of the infinite loop space $\TopO_{(p)}$ are the $p$-torsion subgroups of $\Theta_n$:
\[\pi_n\big(\TopO_{(p)}\big)=\Theta_n\otimes\Z_{(p)}\]
When localized, the short exact sequence $0\rightarrow bP_{n+1}\rightarrow\Theta_n\rightarrow\Theta_n/bP_{n+1}\rightarrow0$ splits. In fact, there is a splitting on the level of spaces: 

\begin{thm}[Localization of $\TopO$]\label{thm:local-splitting}
Let $p$ be an odd prime. Then there are infinite loop spaces  $B=B(p)$ and $C=C(p)$, and infinite loop maps 
\[\beta:B\to\TopO_{(p)}\>\>\>\text{ and }\>\>\>\alpha:C\to\TopO_{(p)}\] 
such that the map 
\[B\times C\xrightarrow{\beta\times\alpha}\TopO_{(p)}\times\TopO_{(p)}\xrightarrow{\mathrm{multiplication}}\TopO_{(p)}\] is an equivalence of infinite loop spaces.
Furthermore, the map $\beta$ induces an isomorphism from $\pi_n(B)$ onto $bP_{n+1}\otimes\Z_{(p)}$.
\end{thm}

This is proved in \cite[\S 5.]{lance-paper}. To compare the statement of Theorem \ref{thm:local-splitting} to what appears in \cite{lance-paper}, one should note that $\PLO_{(p)}\simeq\TopO_{(p)}$ when $p$ is an odd prime because $\TopPL_{(p)}\simeq K(\Z/2\Z,3)_{(p)}$ is contractible. (Localization preserves fibrations of simply connected spaces.) 

\subsection{The Atiyah--Hirzebruch spectral sequence.} The infinite loop space $B=B(p)$ from Theorem \ref{thm:local-splitting} defines spectrum and also a cohomology theory $\mathbb E^*$. In particular, for any space $W$, $\mathbb E^0(W)=[W,B]$. When $W$ is a closed manifold, the groups $\mathbb E^*(W)$ can be computed using the Atiyah--Hirzebruch spectral sequence. This is a spectral sequence with $E_2$ page 
\[E_2^{p,q}=H^p(W;\pi_{-q}(B)),\]
that converges to $\mathbb E^{p+q}(W)$. For general spaces $W$, the spectral sequence converges conditionally, but when $W$ is a closed manifold, the spectral sequence always converges.

\section{Actions on 7-dimensional aspherical space forms (Theorem \ref{thm:dim7})}\label{sec:dim7}

In this section we prove Theorem \ref{thm:dim7}. We briefly explain the strategy. Assuming that $G$ acts freely on $M\#\Sigma$, we use rigidity results (Mostow, Bieberbach, Farrell--Jones, Farrell--Hsiang) to find a free, isometric action of $G$ on $M$ such that, if $\pi:M\rightarrow M/G$ denotes the quotient map, then the smooth structure $M\#\Sigma$ is in the image of the homomorphism 
\[\pi^*:[M/G,\TopO]\rightarrow[M,\TopO].\] 
 
We would like to use this to show that $(M\#\Sigma)/G$ is diffeomorphic to $(M/G)\#\what\Sigma$ for some $\what\Sigma\in\Theta_n$. For this we study $\pi^*$ using the Atiyah--Hirzebruch spectral sequence. Our inability to resolve extension problems in the spectral sequence ultimately forces us to restrict to dimension 7. We argue in two steps, the first of which deals the hyperbolic and flat cases separately.

\subsection*{Step 1: rigidity} In this step we prove Proposition \ref{prop:rigidity}. The proposition does not use the assumption $\dim M=7$ and will be used again for the proof of Theorems \ref{thm:tori} and \ref{thm:asymmetric}. 

\begin{prop}\label{prop:rigidity} 
Fix an exotic $n$-sphere $\Sigma$ and fix an aspherical space form $M$ of dimension $n$. Assume that $M\#\Sigma$ has a faithful, free action of a finite group $G$. 
\begin{enumerate}
\item[(i)] There exists an isometric action of $G$ on $M$ and an equivariant homeomorphism $M\#\Sigma\rightarrow M$.
\item[(ii)] Denoting the quotient map $\pi:M\rightarrow M/G$, then the smooth structure $M\#\Sigma$ is in the image of the homomorphism
\[\pi^*:[M/G,\TopO]\rightarrow[M,\TopO].\]
\end{enumerate} 
\end{prop}

Proposition \ref{prop:rigidity}(i) is similar to part of \cite[Thm.\ 1.5]{CLW}, which says that if $\Isom^+(M)$ acts freely on $M$, then every finite subgroup of $\Homeo^+(M)$ is conjugate into $\Isom^+(M)$. However, in Proposition \ref{prop:rigidity}, we do not assume that the full group $\Isom^+(M)$ acts freely; nevertheless our proof of Proposition \ref{prop:rigidity} is similar to the corresponding part of the proof of \cite[Thm.\ 1.5]{CLW}. 

\begin{proof}[Proof of Proposition \ref{prop:rigidity}] We treat the hyperbolic and flat cases separately, but the proofs are similar.  

\textit{Hyperbolic case:} Assume first that $M$ is a hyperbolic manifold. The action $G\curvearrowright M\#\Sigma$ induces a homomorphism 
\[G\rightarrow\Out\big(\pi_1(M\#\Sigma)\big),\] which is injective by \cite{borel-isometry-aspherical}. By Mostow rigidity $\Out(\pi_1(M))\cong\Isom(M)$, so we obtain an isometric action of $G$ on $M$. To prove $(i)$, we construct an equivariant homeomorphism $M\#\Sigma\to M$. 

{\it Claim.} $G$ acts freely on $M$. 

{\it Proof of Claim.} Consider the following pullback diagram induced by the inclusion $G\hookrightarrow\Out(\pi_1(M))$. 
\begin{equation}\label{eqn:universal-ext}
\begin{xy}
(-80,0)*+{1}="A";
(-60,0)*+{\pi_1(M)}="B";
(-30,0)*+{\Aut(\pi_1(M))}="C";
(0,0)*+{\Out(\pi_1(M))}="D";
(20,0)*+{1}="E";
(-80,-15)*+{1}="F";
(-60,-15)*+{\pi_1(M)}="G";
(-30,-15)*+{\Gamma}="H";
(0,-15)*+{G}="I";
(20,-15)*+{1}="J";
{\ar"A";"B"}?*!/_3mm/{};
{\ar "B";"C"}?*!/_3mm/{};
{\ar "C";"D"}?*!/^3mm/{};
{\ar "D";"E"}?*!/^3mm/{};
{\ar"F";"G"}?*!/_3mm/{};
{\ar "G";"H"}?*!/_3mm/{};
{\ar "H";"I"}?*!/^3mm/{};
{\ar "I";"J"}?*!/^3mm/{};
{\ar@{=} "B";"G"}?*!/_3mm/{};
{\ar "H";"C"}?*!/_3mm/{};
{\ar@{^{(}->} "I";"D"}?*!/_3mm/{};
\end{xy}\end{equation}
By Mostow rigidity, the extension in the top row in the Diagram (\ref{eqn:universal-ext}) is equivalent to the extension
\[1\rightarrow\pi_1(M)\rightarrow N(\pi_1(M))\rightarrow\Isom(M)\rightarrow 1,\] where $N(\pi_1(M))$ denotes the normalizer in $\Isom(\mathbb H^n)$. Consequently, $\Gamma$ can be identified with the group of all lifts of isometries of $G<\Isom(M)$ to the universal cover $\mathbb H^n$. Then $G$ acts freely on $M$ if and only if $\Gamma$ is torsion-free. We prove the Claim by showing $\Gamma$ is torsion free. 

To show that $\Gamma$ is torsion free, we give another description of the extension of $G$ in (\ref{eqn:universal-ext}). Recall that $G$ acts freely on $M\#\Sigma$, so the quotient $M\#\Sigma\rightarrow (M\#\Sigma)/G$ is a covering map, which determines an extension 
\[1\rightarrow\pi_1(M\#\Sigma)\rightarrow\pi_1\big((M\#\Sigma)/G\big)\rightarrow G\rightarrow 1.\]
By construction, the homomorphism $G\to\Out\big(\pi_1(M\#\Sigma)\big)$ that classifies this extension, induces an isomorphism of extensions 
\[\begin{xy}
(-70,0)*+{1}="A";
(-50,0)*+{\pi_1(M\#\Sigma)}="B";
(-20,0)*+{\pi_1\big((M\#\Sigma)/G\big)}="C";
(10,0)*+{G}="D";
(30,0)*+{1}="E";
(-70,-15)*+{1}="F";
(-50,-15)*+{\pi_1(M)}="G";
(-20,-15)*+{\Gamma}="H";
(10,-15)*+{G}="I";
(30,-15)*+{1}="J";
{\ar"A";"B"}?*!/_3mm/{};
{\ar "B";"C"}?*!/_3mm/{};
{\ar "C";"D"}?*!/^3mm/{};
{\ar "D";"E"}?*!/^3mm/{};
{\ar"F";"G"}?*!/_3mm/{};
{\ar "G";"H"}?*!/_3mm/{};
{\ar "H";"I"}?*!/^3mm/{};
{\ar "I";"J"}?*!/^3mm/{};
{\ar "B";"G"}?*!/_3mm/{\cong};
{\ar "C";"H"}?*!/_3mm/{\cong};
{\ar@{=} "D";"I"}?*!/_3mm/{};
\end{xy}\]
The group $\pi_1\big((M\#\Sigma)/G\big)$ is torsion-free because $(M\#\Sigma)/G$ is a closed aspherical manifold. Thus $\Gamma$ is torsion free. This completes the proof of the Claim. 

It remains to obtain an equivariant homeomorphism $M\#\Sigma\rightarrow M$. By the Borel conjecture for hyperbolic manifolds \cite[Cor.\ 10.5]{FJ-borel-hyperbolic}, the isomorphism $\pi_1\big((M\#\Sigma)/G\big)\cong\Gamma\cong\pi_1(M/G)$ is induced by a homeomorphism $f:(M\#\Sigma)/G\cong M/G$. By construction, this homeomorphism lifts to an equivariant homeomorphism $\wtil f:M\#\Sigma\rightarrow M$, as desired. This completes the proof of $(i)$.

In order to prove $(ii)$, we note from part $(i)$ that we have the a commutative diagram 
\begin{equation}\label{eqn:rigidity-hyperbolic}
\begin{xy}
(10,0)*+{M}="B";
(10,-15)*+{M/G}="D";
(-20,0)*+{M\#\Sigma}="E";
(-20,-15)*+{(M\#\Sigma)/G}="F";
{\ar "B";"D"}?*!/_3mm/{\pi};
{\ar "E";"F"}?*!/_3mm/{};
{\ar "E";"B"}?*!/_3mm/{\wtil f};
{\ar "F";"D"}?*!/_3mm/{f};
\end{xy}\end{equation}
The vertical maps are covering maps, and the horizontal maps are homeomorphisms. 
Consider the induced map 
\[\pi^*:S(M/G)\cong[M/G,\TopO]\xrightarrow{}[M,\TopO]\cong S(M)\]
Observe that $\pi^*$ sends an arbitrary element $(W,g)\in S(M/G)$ to $(\what W,\what g)\in S(M)$, where $\what W$ is the pullback covering space with the smooth structure such that $\what W\rightarrow W$ is smooth; see the following diagram. 
\[\begin{xy}
(-10,0)*+{\what W}="A";
(15,0)*+{M}="B";
(-10,-15)*+{W}="C";
(15,-15)*+{M/G}="D";
{\ar"A";"B"}?*!/_3mm/{\what g};
{\ar "B";"D"}?*!/_3mm/{\pi};
{\ar "A";"C"}?*!/^3mm/{};
{\ar "C";"D"}?*!/^3mm/{g};
\end{xy}\]
From this we deduce that $\pi^*$ sends $\big((M\#\Sigma)/G, f)$ to $(M\#\Sigma,\wtil f)$, which shows $M\#\Sigma$ is in the image of $\pi^*$. 

\textit{Flat case:} Assume now that $M$ is flat and let $\Gamma$ denote its fundamental group. The proof is similar to the hyperbolic case, but it uses different rigidity results. Since $G$ acts freely on $M^n\#\Sigma$, the quotient $M^n\#\Sigma\to(M^n\#\Sigma)/G$ is a covering map, so there is an exact sequence
\[1\rightarrow\Gamma\rightarrow\pi_1\big((M\#\Sigma)/G\big)\rightarrow G\rightarrow1.\]
By \cite[Thm.\ 1]{auslander-kuranishi}, the group $\pi_1\big((M\#\Sigma)/G\big)$ is the fundamental group of a flat manifold $\ov M$. Consider the corresponding extension of $\pi_1(\ov M)$. 
\[\begin{xy}
(-80,0)*+{1}="A";
(-60,0)*+{\Gamma}="B";
(-30,0)*+{\pi_1\big((M\#\Sigma)/G\big)}="C";
(0,0)*+{G}="D";
(20,0)*+{1}="E";
(-80,-15)*+{1}="F";
(-60,-15)*+{\Gamma}="G";
(-30,-15)*+{\pi_1(\ov M)}="H";
(0,-15)*+{G}="I";
(20,-15)*+{1}="J";
{\ar"A";"B"}?*!/_3mm/{};
{\ar "B";"C"}?*!/_3mm/{};
{\ar "C";"D"}?*!/^3mm/{};
{\ar "D";"E"}?*!/^3mm/{};
{\ar"F";"G"}?*!/_3mm/{};
{\ar "G";"H"}?*!/_3mm/{};
{\ar "H";"I"}?*!/^3mm/{};
{\ar "I";"J"}?*!/^3mm/{};
{\ar@{=} "B";"G"}?*!/_3mm/{};
{\ar "H";"C"}?*!/^3mm/{\cong};
{\ar@{=} "I";"D"}?*!/^3mm/{};
\end{xy}\]
By the Borel conjecture for flat manifolds \cite{farrell-hsiang}, the isomorphism $\pi_1(M/G)\cong\pi_1(\ov M)$ above is induced by a homeomorphism $f:(M\#\Sigma)/G\to\ov M$. By construction, this homeomorphism lifts to an equivariant homeomorphism $\wtil f:M\#\Sigma\rightarrow M$ between the covering spaces with fundamental group $\Gamma$ (the cover of $\ov M$ is $M$ by Bieberbach's second theorem \cite[Ch.\ I, Thm.\ 4.1]{charlap}). By construction, the codomain $M$ has a flat metric induced from $\ov M$ on which $G$ acts isometrically. This proves $(i)$. 

To show $(ii)$, we consider the following diagram 
\begin{equation*}
\begin{xy}
(10,0)*+{M}="B";
(10,-15)*+{\ov M}="D";
(-20,0)*+{M\#\Sigma}="E";
(-20,-15)*+{(M\#\Sigma)/G}="F";
{\ar "B";"D"}?*!/_3mm/{\pi};
{\ar "E";"F"}?*!/_3mm/{};
{\ar "E";"B"}?*!/_3mm/{\wtil f};
{\ar "F";"D"}?*!/_3mm/{f};
\end{xy}\end{equation*}
As in the proof of the previous case, from this diagram, we conclude that the homomorphism
\[\pi^*:S(\ov M)\cong[\ov M,\TopO]\xrightarrow{}[M,\TopO]\cong S(M)\]
sends $\big((M\#\Sigma)/G,f\big)$ to $(M\#\Sigma,\wtil f)$, so $M\#\Sigma$ is in the image of $\pi^*$. 

This completes the proof of Proposition \ref{prop:rigidity}.
\end{proof}

\subsection*{Step 2: computation with the Atiyah--Hirzebruch spectral sequence} 

Consider the Atiyah--Hirzebruch spectral sequence for the cohomology theory $\mathbb{E}^*$ determined by the infinite loop space $\TopO$, so that $\mathbb{E}^0(M)=[M,\TopO]$. This spectral sequence has $E_2$-page
\[E_2^{p,q}=H^p(M;\pi_{-q}(\TopO)).\]
Recall that $\pi_k(\TopO)=\Theta_k$ if $k\neq 3$ and $\pi_3(\TopO)\cong\Z/2\Z$. In particular
$\pi_k(\TopO)=0$ for $k\in\{0,1,2,4,5,6\}$. Thus the spectral sequence   yields an exact sequence
\[H^7(M;\Theta_7)\xrightarrow{c}[M,\TopO]\xrightarrow{q} H^3(M;\Z/2\Z)\rightarrow0\]
Identifying $H^7(M;\Theta_7)\cong\Theta_7$, the homomorphism $c$ is the map $\Sigma\mapsto(M\#\Sigma,\iota)$. (Aside: We will not need to interpret the homomorphism $q$, but it can be identified with the natural map $[M,\TopO]\rightarrow[M,\TopPL]$ that sends a smooth structure to its corresponding PL structure.)

The map $\pi:M\rightarrow M/G$ induces a map of spectral sequences and a commutative diagram
\beq
\xymatrix{
\Theta_7\ar[r]^-{c}\ar[d]^{\cdot |G|} & [M/G,\TopO]\ar[d]^{\pi^*}\ar[d]^{\pi^*}\ar[r]^-{q} &H^3(M/G;\Z/2\Z)\ar[d]^{\pi^*}\\
\Theta_7\ar[r]^-{c}&[M,\TopO]\ar[r]^-{q}&H^3(M;\Z/2\Z)
}
\eeq

The left vertical map is multiplication by $|G|$, which is the degree of the cover $\pi$. As $|G|$ is odd, the right vertical map is injective by a transfer argument \cite[Prop.\ 3G.1]{hatcher}. 

To complete the proof of Theorem \ref{thm:dim7}, we note that by Step 1, $M\#\Sigma=\pi^*(x)$, where $x=((M\#\Sigma)/G,f)$; c.f.\ Diagram (\ref{eqn:rigidity-hyperbolic}).  Since $M\#\Sigma=c(\Sigma)$, we have $0=q(\pi^*(x))=\pi^*(q(x))$, which implies that $q(x)=0$ since the right vertical map is injective. Thus $x=c(\what\Sigma)$ for some $\Sigma\in\Theta_7$, i.e.\ $(M\#\Sigma)/G$ is concordant, hence diffeomorphic (by \cite[Essay I, Theorem 4.1]{kirby-siebenmann}), to $(M/G)\#\what\Sigma$. By Proposition \ref{prop:rigidity}(i), the action $G\curvearrowright M\#\Sigma$ is topologically conjugate to an isometric action on $M$. Therefore, we can apply Lemma \ref{lem:standard-actions} to conclude that the action of $G$ on $M\#\Sigma$ is standard.

To prove Addendum \ref{add:hyperbolic7} we use the fact, shown in \cite{kupers}, that for all closed $7$-manifolds $W$ there is a splitting  
\beq
[W,\TopO]\cong H^7(W;\Theta_7)\oplus H^3(W;\Z/2\Z).
\eeq
With respect to this splitting the map $\pi^*$ is diagonal. Hence writing 
$$M\#\Sigma=(a,0)\in H^7(M;\Theta_7)\oplus H^3(M;\Z/2\Z)$$
with $a\neq 0$ and
$$x=(u,v)\in H^7(M;\Theta_7)\oplus H^3(M;\Z/2\Z),$$ 
 we have
$\pi^*(u,v)=(a,0)$ and $u\neq 0$. Since $\pi$ is a degree $p$ map, we have $\pi^*(u,0)=(p\cdot u,0)$ and the divisibility property follows. \qed
\section{Two facts about flat manifolds} \label{sec:flat}

In this section we prove two results that are needed in preparation for the proof of Theorem \ref{thm:tori}. These results (Propositions \ \ref{prop:mapping-torus} and \ref{prop:parallelizable}) concern the structure of flat manifolds with prime cyclic holonomy. 

\begin{prop}[Prime cyclic holonomy implies mapping torus]\label{prop:mapping-torus}
Fix a prime $p$. Let $N$ be a compact flat manifold with holonomy group $G=\Z/p\Z$. Then $N$ has the structure of a mapping torus 
\[T^{n-1}\rightarrow N\rightarrow S^1\] 
whose monodromy has order $p$. 
\end{prop}

\begin{proof}[Proof of Proposition \ref{prop:mapping-torus}]

By assumption $\pi_1(N)$ is isomorphic to a discrete torsion-free subgroup $\Gamma<\Isom^+(\R^n)\cong\R^n\rtimes\SO(n)$, and $\Gamma$ is torsion-free because $N$ is an aspherical manifold (not an orbifold). 

{\it Step $1$.} First we construct a surjection $\Gamma\to\Z$ with kernel $\Z^{n-1}$. This step uses an argument of \cite[Lem.\ 1]{cliff-weiss}. 

Let $A<\Gamma$ be the maximal abelian subgroup, $A\cong\Z^n$. There is a short exact sequence 
\begin{equation}\label{eqn:extension}1\rightarrow A\rightarrow\Gamma\xrightarrow{} G\rightarrow 1.\end{equation}
Let $\xi\in H^2(G;A)$ denote the Euler class of the extension (c.f.\ \cite[Ch.\ IV, Thm.\ 3.12]{brown}), where the coefficients $A$ has the $\Z[G]$-module structure coming from the extension (\ref{eqn:extension}). We know that $\xi\neq0$ because otherwise $\Gamma\cong\Z^n\rtimes G$ is not torsion-free. 

Fix a generator $g$ of $G$, and set 
\[\delta=1+g+\cdots+g^{p-1}\in\Z[G].\]
For a fixed $\Z[G]$ module $M$, write $m_\delta:M\to M$ for the homomorphism defined by 
$m_\delta(x)=\delta\cdot x$. Recall that  
\begin{equation}\label{eqn:H2}H^2(G;M)\cong M^G/\im\big[m_\delta:M\to M\big]\end{equation} by the standard resolution for computing cohomology of cyclic groups (c.f.\ \cite[I.6]{brown}). 

Set 
\[A'=\ker\big[m_\delta:A\to A\big]\>\>\>\text{ and }\>\>\>A''=\im\big[m_\delta:A\to A\big],\] so there is a short exact sequence of $\Z[G]$ modules
\[0\rightarrow A'\rightarrow A\rightarrow A''\rightarrow0.\]
Consider the following portion of the associated long exact sequence in cohomology with coefficients
\[H^2(G;A')\rightarrow H^2(G;A)\rightarrow H^2(G;A'')\]
Observe that $(A')^G=0$ (if $a\in A'$ and $g\cdot a=a$, then $0=m_\delta(a)=p\cdot a$). Then 
\[H^2(G;A')=0\] by (\ref{eqn:H2}), which implies that $H^2(G;A)\rightarrow H^2(G;A'')$ is injective. Furthermore, $A''$ is a trivial module (because $g\cdot\delta=\delta$), so 
\[H^2(G;A'')\cong H^2(G;\Z)\oplus\cdots\oplus H^2(G;\Z).\] 
Since $\xi\neq0$, there is surjective composition $q:A\onto A''\onto\Z$ so that the image of $\xi$ under the induced homomorphism 
\[q_*:H^2(G;A)\hookrightarrow H^2(G;A'')\rightarrow H^2(G;\Z)\]
is nonzero. 

Set $B=\ker(q)$. This is a $\Z[G]$-submodule of $A$, which implies that $B$ is a normal subgroup of $\Gamma$. By construction, the extension 
\begin{equation}\label{eqn:Gamma-mod-B}1\rightarrow A/B\rightarrow\Gamma/B\rightarrow\Gamma/A\rightarrow 1\end{equation}
does not split (its Euler class is $q_*(\xi)\neq0$). Here $\Gamma/A\cong G\cong\Z/p\Z$ and $A/B\cong\Z$ so the preceding central extension has the form
\[1\to\Z\to\Gamma/B\to\Z/p\Z\to1.\]
This implies $\Gamma/B\cong\Z$ . 
 
Thus we have a surjection $\Gamma\to\Gamma/B\cong\Z$ with kernel $B\cong\Z^{n-1}$, as desired. 

{\it Step $2$.} We explain why the short exact sequence $1\to B\to\Gamma\to\Z\to1$ from Step 1 is realized topologically as a fiber bundle $T^{n-1}\to N\to S^1$. For this, it suffices to show that $N$ is obtained from a quotient of $T^{n-1}\times\R$ by a homeomorphism of the form \begin{equation}\label{eqn:monodromy}h(x,t)=(h_1(x),t+t_0)\end{equation}
for some $h_1\in\Homeo(T^{n-1})$ and $t_0\neq0\in\R$. 

Consider the action of $\Gamma$ on $\R^n$. Let $V\subset\R^n$ be the subspace spanned by the orbit of $0\in\R^n$ under $B$. Decompose $\R^n=V\oplus V^{\perp}$. Since $B$ is normal in $\Gamma$, the action of $\Gamma$ on $\R^n$ descends to an action of $\Gamma/B\cong\Z$ on 
$\R^n/B\cong V/B\times V^\perp\cong T^{n-1}\times\R$; furthermore, $N$ is the quotient of this action. 

Fix $\gamma\in\Gamma$ that projects to a generator under $\Gamma\to\Z$. Write $\gamma=(u,f)\in\R^n\rtimes\SO(n)$. The element $f\in\SO(n)$ generates the image of the holonomy $\Gamma\to\SO(n)$, so $f$ has order $p$. Also, $\gamma$ normalizes $B$, so $f$ preserves $V$ and the decomposition $\R^n=V\oplus V^{\perp}$. Write $f=f_1\oplus f_2$ for the action on $V\oplus V^{\perp}$, and write $u=(u_1,u_2)\in V\oplus V^{\perp}$. Since $f_2$ is an isometry of $V^{\perp}\cong\R$, $f_2=\pm1$. It suffices to show $f_2=1$ because then $\gamma$ acts on $V\oplus V^\perp$ by
\begin{equation}\label{eqn:monodromy2}\gamma(x,t)=(f_1(x)+u_1, t+u_2)\end{equation}
and this descends to an action on $\R^n/B\cong T^{n-1}\times\R$ of the form (\ref{eqn:monodromy}). 

Suppose for a contradiction that $f_2=-1$. Since $f$ has order $p$, this is only possible when $p=2$. Then the (induced) action of $\gamma$ on $T^{n-1}\times\R$ acts by a reflection in the $\R$ direction ($\gamma$ preserves the foliation $T^{n-1}\times\{t\}$, so this action on $\R$ is well-defined). But then the quotient $N=\big(T^{n-1}\times\R\big)/\langle\gamma\rangle$ is not compact, which is a contradiction. 

{\it Step $3$.} Finally we explain why we can assume the monodromy of $T^{n-1}\to N\to S^1$ has order $p$. From the equation (\ref{eqn:monodromy2}), the monodromy is the map $x\mapsto f_1(x)+u_1$ on $V/B$. This map is isotopic to $x\mapsto f_1(x)$ since $x\mapsto x+u_1$ is a rotation on $T^{n-1}\cong V/B$ and rotations are isotopic to the identity. Since $f_1$ has order $p$, this completes the proof. 
  
\end{proof}

\begin{rmk}
Our proof does not work if $G=\Z/d\Z$ when $d$ is not prime. One reason is that when $d$ is not prime there are non-split extensions $1\to\Z\to\Lambda\to\Z/d\Z\to1$ with $\Lambda\neq\Z$. In this case it seems the only possible conclusion would be that $N$ is a mapping torus $N_1\to N\to S^1$ where $N_1$ is a $(n-1)$-dimensional flat manifold. 
\end{rmk}

\begin{prop}[Flat mapping torus is parallelizable]\label{prop:parallelizable}
Let $M_f$ be a flat manifold that has the structure of a mapping torus $T^{n-1}\rightarrow M_f\rightarrow S^1$ whose monodromy $f$ is orientation-preserving and has finite order. Then $M_f$ is parallelizable. 
\end{prop}

A version of this is proved by \cite{thorpe} with a different assumption. We give an alternate argument.

\begin{proof}[Proof of Proposition \ref{prop:parallelizable}]
To show $M_f$ is parallelizable, it suffices to construct an $n$-frame field on $\widetilde M_f\cong\R^n$ that is $\Gamma$-invariant, where $\Gamma=\pi_1(M_f)$. Write $\Gamma=\Z^{n-1}\rtimes_{f_*}\Z$, and decompose $\R^n$ accordingly as $\R^n=\R^{n-1}\times\R$. 

The group $\Gamma$ has a generating set $\gamma_1,\ldots,\gamma_{n-1},\eta$, where $\langle\gamma_1,\ldots,\gamma_{n-1}\rangle\cong\Z^{n-1}$ acts by translations of $\R^{n-1}\times\R$ that are trivial in the second factor, and $\eta$ acts on $(x,t)\in\R^{n-1}\times\R$ by $\eta(x,t)=(f_*(x)+\beta,t+\frac{1}{d})$, where $(\beta,f_*)\in\R^{n-1}\rtimes \SO(n-1)$.  

Define an orthonormal $n$-frame field on $\R^n$ as follows. First define a frame field along $\R^{n-1}\times0$ by choosing an orthonormal frame at the origin, and moving it along $\R^{n-1}$ by parallel transport. Choose this frame compatible with the decomposition $\R^n=\R^{n-1}\times\R$ (i.e.\ for the frame at the origin, $n-1$ vectors belong to the subspace $\R^{n-1}$ and the last vector belongs to the subspace $\R$). 

Now let $\alpha_t$ be a path from the identity $I$ to $f_*$ in $\SO(n-1)$, defined for $t\in[0,1/d]$.  Define a frame on $\R^{n-1}\times\{t\}$ by acting by $\alpha_t$ on the framing of $\R^{n-1}\times\{0\}$.

This extends in an obvious way to a framing on $\R^{n-1}\times\R$ that is $\eta$-invariant. The resulting framing is $\gamma_i$-invariant for each $i$, since by construction it is constant on $\R^{n-1}\times\{t\}$ for each $t\in\mathbb R$. Since $\Gamma$ is generated by $\eta$ and the $\gamma_i$, the framing is $\Gamma$ invariant. 
\end{proof}

\section{Actions on exotic tori (Theorem \ref{thm:tori})}\label{sec:tori}
In this section we prove Theorem \ref{thm:tori}. 
We begin with a brief sketch of the argument, and then give the details in the following subsections. 

Fix $\Sigma\in\Theta_n$ and a suppose that $G\cong\Z/p\Z$ acts freely on $T^n\#\Sigma$ by orientation-preserving diffeomorphisms. As $T^n$ is a flat manifold, Proposition \ref{prop:rigidity} gives an action $G\curvearrowright T^n$ such that $T^n\#\Sigma$ is in the image of a homomorphism
\[\pi^*:[Q,\TopO]\rightarrow[T^n,\TopO],\]
induced by the quotient $\pi:T^n\to T^n/G=:Q$. 

We show that there are isomorphisms
\begin{equation}\label{eqn:compatible-splitting}\begin{xy}
(-25,0)*+{[T^n,\TopO]}="A";
(13,0)*+{[S^n,\TopO]\>\>\>\oplus\>\>\>A}="B";
(-25,-15)*+{[Q,\TopO]}="C";
(13,-15)*+{[S^n,\TopO]\>\>\>\oplus\>\>\>\bar{A}}="D";
(-10,0)*+{\cong}="E";
(-10,-15)*+{\cong}="F";
{\ar "C";"A"}?*!/_3mm/{\pi^*};
\end{xy}\end{equation}
where $A$ and $\bar{A}$ are finite abelian groups. With respect to this splitting, we show that the homomorphism $\pi^*$ is diagonal; see Corollary \ref{cor:splitting}. 

Given this, the proof is completed as follows. With respect the isomorphisms in (\ref{eqn:compatible-splitting}), we write $T^n\#\Sigma$ as $(x,0)\in[S^n,\TopO]\oplus A$. From Proposition \ref{prop:rigidity}, we have $(x,0)=\pi^*(y,a)$. Using that $\pi^*$ is diagonal, we deduce that $\pi^*(0,a)=0$ and $y\neq0$; furthermore, since $\pi$ is a degree-$p$ covering map, $\pi^*(y,0)=(py,0)$. This implies the divisibility property. 

\subsection{Compatible splitting of the top cell} We now provide details for the existence of a ``diagonal'' homomorphism as in (\ref{eqn:compatible-splitting}).

For an open embedding $e:X\hookrightarrow Y$ of manifolds, we denote $e':Y'\rightarrow X'$ the induced map of 1-point compactifications. Recall that a smooth $n$-manifold $X$ is \textit{stably parallelizable} if it can be smoothly embedded in $\R^{n+k}$, for some $k\geq 1$, with trivial normal bundle. 

The following lemma is well known. We give a proof, which will be helpful in preparation for the Proposition \ref{prop:compatible-splitting}. 

\begin{lem}[Splitting the top cell]\label{lem:splitting}
Let $W$ be a stably parallelizable closed $n$-manifold. Then for $k\ge n$, there is a homotopy equivalence $ S^{n+k}\vee Z_W\to \Sigma^kW$, where $Z_W$ is a finite $CW$-complex of dimension $< n+k$. 
\end{lem}

\begin{proof}
Give $W$ a cell structure with a single $n$-cell, and give $\Sigma^kW$ the induced cell structure, which has a single $(n+k)$-cell. Let $Z=(\Sigma^kW)^{(n+k-1)}$ denote the $(n+k-1)$-skeleton. There is a cofiber sequence 
\[Z\xrightarrow{i} \Sigma^kW\xrightarrow{q} S^{n+k}\]

Since $W$ is stably parallelizable, there is a framed Whitney embedding $j:W\times D^k\hookrightarrow\R^{n+k}$. The map 
\[p:S^{n+k}\xrightarrow{j'}\Sigma^k(W_+)\rightarrow\Sigma^kW\] is a right inverse to the map $q$ (up to homotopy) since the composition $q\circ p:S^{n+k}\rightarrow S^{n+k}$ has degree 1. 

{\it Claim.} The map $p\vee i: S^{n+k}\vee Z\rightarrow \Sigma^kW$ is a homotopy equivalence. 

{\it Proof of Claim.} Since the domain and codomain are simply connected, it suffices to show that $p\vee i$ is a homology equivalence by Whitehead's theorem \cite[Cor.\ 4.33]{hatcher}. It is easy to see that $p\vee i$ induces an isomorphism on $H_\ell$ for $\ell\le n+k-2$ since $Z$ is the $(n+k-1)$ skeleton; see e.g.\ \cite[Lem.\ 2.34(c)]{hatcher}. It remains to treat the cases $\ell=n+k-1$ and $\ell=n+k$. 

For $\ell=n+k$, the composition $S^{n+k}\vee Z\xrightarrow{p\vee i}\Sigma^kW\xrightarrow{q} S^{n+k}$ induces an isomorphism on $H_{n+k}$ (since $q\circ p$ has degree 1), and since each of these spaces has $H_{n+k}=\Z$, it follows that $(p\vee i)_*:H_{n+k}(S^{n+k}\vee Z)\rightarrow H_{n+k}(\Sigma^kW)$ is an isomorphism (and also that $q_*:H_{n+k}(\Sigma^kW)\rightarrow H_{n+k}(S^{n+k})$ is an isomorphism). 

For $\ell=n+k-1$, it suffices to show that $i_*:H_{n+k-1}(Z)\rightarrow H_{n+k-1}(\Sigma^kW)$ is injective. Considering the long-exact sequence of the pair $(\Sigma^kW,Z)$, it is equivalent to show that the homomorphism $H_{n+k}(\Sigma^kW)\rightarrow H_{n+k}(\Sigma^kW,Z)$ is surjective. This homomorphism can be identified with $q_*:H_{n+k}(\Sigma^kW)\rightarrow H_{n+k}(\Sigma^kW/Z)$, which we observed is an isomorphism in the preceding paragraph. 

This proves the claim, and finishes the proof of the lemma. 
\end{proof}

\begin{prop}[Compatible splitting]\label{prop:compatible-splitting}
Let $\pi:T^n\rightarrow Q$ be the quotient by a free action of $G=\Z/p\Z\curvearrowright T^n$. For $k\ge n$ there exists splittings $\Sigma^kT^n\simeq S^{n+k}\vee Z_{T^n}$ and $\Sigma^kQ\simeq S^{n+k}\vee Z_{Q}$ as in Lemma \ref{lem:splitting} so that the map 
\begin{equation}\label{eqn:cross-term}S^{n+k}\hookrightarrow S^{n+k}\vee Z_{T^n}\xrightarrow{\simeq}\Sigma^kT^n\xrightarrow{\Sigma^k(\pi)}\Sigma^kQ\xrightarrow{\simeq} S^{n+k}\vee Z_Q\onto Z_Q,\end{equation}
where the fourth map is a homotopy inverse to the splitting of Lemma \ref{lem:splitting},
is homotopically trivial. 
\end{prop}

\begin{proof}[Proof of Proposition \ref{prop:compatible-splitting}]

By Proposition \ref{prop:mapping-torus}, $Q$ has the structure of a mapping torus $T^{n-1}\to Q\to S^1$ with order-$p$ monodromy $f$. Set $H=\langle f\rangle$. 

Fix an equivariant embedding $T^{n-1}\hookrightarrow V$ into a $\R[H]$-module $V$; this is possible by an equivariant version of the Whitney embedding theorem \cite{palais,mostow}. Without loss of generality, we can assume that the action $f\curvearrowright V$ is orientation-preserving (if not, we can replace $V$ by $V\oplus V$). 

Let $f_t$ be a homotopy in $\SO(V)$ from $f$ to the identity. We assume this homotopy takes place for $t\in(0,\infty)$, with $f_t=\id$ for $t\le 1$ and $f_t=f$ for $t\ge2$.

The map $F(v,t)=(f_t(v),t)$ is a homeomorphism of $V\times(0,\infty)$. Let $M_F$ be the mapping torus of $F$. We can view $M_F$ as a bundle 
\[V\rightarrow M_F\rightarrow \R^2\setminus\{0\}.\] Since the bundle is trivial near $0$, we can extend $M_F$ to a bundle $V\rightarrow \what M_F\rightarrow\R^2$. This latter bundle is trivial $\what M_F\cong V\times\R^2$ since $\R^2$ is contractible. By construction, $Q=M_f$ embeds in $M_F\subset \what M_F$. (Explicitly, choose the copy of $T^{n-1}$ in $V\times\{2\}$. Then the image of $T^{n-1}\times\{2\}\times[0,1]$ in $M_F=\big(V\times(0,\infty)\times[0,1]\big)/\sim$ is an embedding of $M_f$ in $M_F$.)

Now we lift to an equivariant embedding of $T^n$. Identify $\R^2\cong\C$, and consider the $p$-fold cover 
\[\phi:V\times\C\rightarrow V\times\C\>\>\>\text{ given by $\>\>\>(v,z)\mapsto(v,z^p)$.}\] 
This is a regular cover, branched over $V\times\{0\}$ with deck group $G=\Z/p\Z$. The pre-image of $M_f$ under this cover is $M_{f^p}=M_{\id}=T^n$. Thus we have an embedding $T^n\hookrightarrow V\times\R^2$ that is equivariant with respect to the deck group actions of the coverings $\pi:T^n\rightarrow M_f$ and $\phi:V\times\R^2\rightarrow V\times\R^2$, and the quotient by these actions yields an embedding $Q=M_f\hookrightarrow V\times\R^2$. Since $T^n$ and $Q$ are stably parallelizable by Proposition \ref{prop:parallelizable}, the normal bundles of $T^n\subset V\times\R^2$ and $Q\subset V\times\R^2$ are trivial. (For this it's possible we need to first take product with $\R^m$ (with trivial $G$ action); c.f.\ \cite[Lem.\ 3.3]{kervaire-milnor}.) 
Thus we have a commutative diagram:  
\beq
\xymatrix{
T^n\times D^k\ar@{^{(}->}[r]^-{j_{T^n}}\ar[d]_{\pi} & V\times\R^2\ar[d]^{\phi}\\
Q\times D^k\ar@{^{(}->}[r]^-{j_{Q}} & V\times\R^2
}
\eeq

We use collapse maps of the embeddings $j_{T^n},j_Q$ to get a commutative diagram 
\begin{equation}\label{eqn:compatible-splitting2}
\xymatrix{
S^{n+k}\ar[r]^{p_{T^n}}\ar[d]_{\phi'} & \Sigma^kT^n\ar[d]^{\Sigma^k(\pi)}\\
S^{n+k}\ar[r]^{p_Q} &\Sigma^kQ
}
\end{equation}

By the proof of Lemma \ref{lem:splitting}, for $W=T^n$ or $Q$, the map $p_W:S^{n+k}\rightarrow\Sigma^kW$ splits the top cell, i.e.\ there is a homotopy equivalence so $\Sigma^kW\simeq S^{n+k}\vee Z_W$ so that the composition
\begin{equation}\label{eqn:homotopically-trivial}S^{n+k}\xrightarrow{p_W}\Sigma^kW\simeq S^{n+k}\vee Z_W\onto Z_W\end{equation} is homotopically trivial. Then the composition (\ref{eqn:cross-term}) is homotopically trivial because it factors through (\ref{eqn:homotopically-trivial}) by virtue of diagram (\ref{eqn:compatible-splitting2}). 
\end{proof}

We now use Proposition \ref{prop:compatible-splitting} to prove that there is a ``compatible splitting" of 
\[\pi^*:[Q,\TopO]\rightarrow[T^n,\TopO].\] Let $\pi:T^n\rightarrow Q$ be the quotient by a free action of $G=\Z/p\Z$. Let $u_Q:D^n\hookrightarrow Q$ be an embedded disk, chosen sufficiently small so that it lifts to an embedding $\tilde u:\sqcup_pD^n\hookrightarrow T^n$. Let $u_{T^n}:D^n\hookrightarrow T^n$ be an embedded disk that contains the image of $\tilde u$, so $\tilde u$ factors as $\sqcup_p D^n\xrightarrow{j} D^n\xrightarrow{u_{T^n}} T^n$. This leads to the following commutative diagram of spaces. 
\[\begin{xy}
(40,0)*+{S^n}="A";
(0,0)*+{T^n}="B";
(40,-30)*+{S^n}="C";
(40,-15)*+{\bigvee_pS^n}="D";
(0,-30)*+{Q}="E";
{\ar"B";"A"}?*!/_3mm/{u_{T^n}'};
{\ar "B";"D"}?*!/^3mm/{\tilde u'};
{\ar "A";"D"}?*!/_3mm/{j'};
{\ar "D";"C"}?*!/_3mm/{\Delta};
{\ar "B";"E"}?*!/^3mm/{\pi};
{\ar "E";"C"}?*!/^3mm/{u_Q'};
\end{xy}\]
Here the map $\Delta$ is the identity map on each $S^n$. The composition $\Delta\circ j':S^n\rightarrow S^n$ has degree $p$. Suspending this diagram and combining with Diagram (\ref{eqn:compatible-splitting2}), we obtain
\begin{equation}\label{eqn:splitting-diagram}
\xymatrix{
S^{n+k}\ar[r]^{p_{T^n}}\ar[d]_{\mathrm{deg}=p} & 
\Sigma^kT^n_+\ar[r]^{u'_{T^n}}\ar[d] & S^{n+k} \ar[d]_{}\\
S^{n+k}\ar[r]^{p_{Q}} & \Sigma^kQ_+\ar[r]^{u'_{Q}} & S^{n+k}
}
\end{equation}

The composition $S^{n+k}\rightarrow S^{n+k}$ in each row is a degree-1 map. Recalling that $\TopO$ is an infinite loop space, let $Y$ be a space such that $\Omega^kY\simeq\TopO$. Apply $[-,Y]$ to Diagram (\ref{eqn:splitting-diagram}), and use the adjunction $[A,\Omega B]\cong[\Sigma A,B]$ to arrive at the following diagram. 
\beq
\xymatrix{
[S^n,\TopO]\ar[r]^{(u'_Q)^*} \ar[d]_{}&[Q,\TopO]\ar[r]^{p^*_Q}\ar[d]^{\pi^*} & [S^n,\TopO]\ar[d]^{\mu_p}\\
[S^n,\TopO]\ar[r]^{(u'_{T^n})^*} &[ T^n,\TopO]\ar[r]^{p^*_{T^n}} &[S^n,\TopO]
}
\eeq

Here $\mu_p$ is multiplication by $p$. We have proved the following corollary. 

\begin{cor}\label{cor:splitting}
The maps $(u'_Q)^*$ and $(u'_{T^n})^*$ are split injections, and there exist splittings $p_Q^*,p_{T^n}^*$ with the property that 
\begin{equation}\label{eqn:compatible}p_{T^n}^*\circ\pi^*=\mu_p\circ p_Q^*.\end{equation}
\end{cor}

\subsection{Finishing the proof of Theorem \ref{thm:tori}}

In Proposition \ref{prop:rigidity}(ii), we showed that $T^n\#\Sigma=\pi^*(z)$ for some $z\in[Q,\TopO]$. Using Corollary \ref{cor:splitting}, we have 
\[\Sigma=p_{T^n}^*(T^n\#\Sigma)=p_{T^n}^*(\pi^*(z))=\mu_p(p_Q^*(z))\]
which shows that $\Sigma$ is divisible by $p$ in $\Theta_n=[S^n,\TopO]$, as desired. 
 
\section{Asymmetric smoothings (Theorem \ref{thm:asymmetric})}\label{sec:asymmetric} 

Recall from Proposition \ref{prop:rigidity} that if $M$ is a closed hyperbolic manifold and $G$ acts freely on $M\#\Sigma$, then the smoothing $M\#\Sigma$ is in the image of a certain homomorphism
\[\pi^*:[M/G,\TopO]\rightarrow[M,\TopO].\]
We would like to use this to conclude that $M$ satisfies the divisibility property, similar to Step 2 in the proof of Theorem \ref{thm:tori}. If we could show this, then it would suffice to find $M$ and $\Sigma$ with the property that $\Sigma$ is not divisible by $|G|$ for any nontrivial group $G<\Isom^+(M)$. 

Unfortunately, we don't know how to prove that hyperbolic manifolds satisfy the divisibility property in general (it seems difficult to produce a geometric construction that would yield a version of Proposition \ref{prop:compatible-splitting} in the hyperbolic case). Instead, as in the proof of Theorem \ref{thm:dim7}, we study $\pi^*$ using the Atiyah--Hirzebruch spectral sequence. Here the difficulty is (as usual) potentially nontrivial differentials and extension problems, but we show these issues can be avoided for a proper choice of $M,\Sigma$ and by localizing $\TopO$ at an odd prime.  

\begin{proof}[Proof of Theorem \ref{thm:asymmetric}]\

\subsection*{Step 1: the construction} Fix $n_0\geq 5$, $d\ge1$, and choose $n=4k-1\ge n_0$ and an odd prime $p$ such that the $p$-torsion subgroup of $bP_{n+1}$ is nontrivial and the $p$-torsion subgroup of $bP_{m+1}$ is trivial for $m<n$. This is possible because the set of primes that divide $|bP_{4k}|$ for some $k$ is infinite. For example, $|bP_{4k}|$ is divisible by $2^{2k-1}-1$ (see \cite[\S7]{kervaire-milnor}), and it is not difficult to show that if $s,t$ are relatively prime, then $2^s-1$ and $2^t-1$ are relatively prime. 

Let $\Z_{(p)}$ denote the set of rational numbers with denominator relatively prime to $p$. The group $(bP_{n+1})_{(p)}:=bP_{n+1}\otimes\Z_{(p)}$ is the $p$-torsion of $bP_{n+1}$. Since $bP_{n+1}$ is cyclic, $(bP_{n+1})_{(p)}\cong\Z/p^a\Z$ for some $a\ge1$.  Choose a generator $\Sigma\in (bP_{n+1})_{(p)}$. 

Next choose a closed oriented hyperbolic $n$-manifold $M$ such that (i) $\Isom^+(M)=\Isom(M)$, (ii) $\Isom(M)$ is a $p$-group where every element has order divisible by $p^a$, and (iii) $\Isom(M)$ acts freely on $M$. Such examples exist by the construction of Belolipetsky--Lubotzky \cite[Thm.\ 1.1]{belolipetsky-lubotzky}; see also \cite[Thm.\ 6]{BT-hyperbolic}.

\subsection*{Step 2: the computation.} Take $M$ and $\Sigma$ as in Step 1. We claim that $N=M\#\Sigma$ is asymmetric. Fix a finite order element $g\in\Diff^+(M\#\Sigma)$ and denote $G=\pair{g}$. Suppose for a contradiction that $g\neq\id_N$. By a result of Borel \cite{borel-isometry-aspherical}, the induced map $G\rightarrow\Out(\pi_1(N))\cong\Isom(M)$ is injective, so the order of $g$ is $p^b$ for some $b\ge a$.  

By Proposition \ref{prop:rigidity} , there is a degree $|G|$ covering map $\pi:M\rightarrow\ov M$ and $x\in[\ov M,\TopO]$ such that $\pi^*(x)=M\#\Sigma$, where $\pi^*:[\ov M,\TopO]\rightarrow[M,\TopO]$. 

To explain the remainder of the argument, we consider the following commutative diagram. 
\begin{equation}\label{eqn:hyperbolic-diagram}
\begin{xy}
(-70,0)*+{[\ov M,\TopO]}="A";
(-70,-15)*+{[\ov M,\TopO_{(p)}]}="B";
(-70,-30)*+{[\ov M,B]}="C";
(-30,0)*+{[M,\TopO]}="D";
(-30,-15)*+{[M,\TopO_{(p)}]}="E";
(-30,-30)*+{[M,B]}="F";
(10,0)*+{[S^n,\TopO]}="G";
(10,-15)*+{[S^n,\TopO_{(p)}]}="H";
(10,-30)*+{[S^n,B]}="I";
(26,0)*+{\cong\Theta_n}="J";
(34,-15)*+{\cong\Theta_n\otimes\Z_{(p)}}="K";
(30,-30)*+{\cong bP_{n+1}\otimes\Z_{(p)}}="L";
{\ar"A";"B"}?*!/_3mm/{};
{\ar "B";"C"}?*!/_3mm/{};
{\ar "D";"E"}?*!/^3mm/{};
{\ar "E";"F"}?*!/^3mm/{};
{\ar "G";"H"}?*!/^3mm/{};
{\ar "H";"I"}?*!/^3mm/{};
{\ar "A";"D"}?*!/_3mm/{\pi^*};
{\ar "B";"E"}?*!/^3mm/{};
{\ar "C";"F"}?*!/_3mm/{0};
{\ar "G";"D"}?*!/^3mm/{(i')^*};
{\ar "H";"E"}?*!/^3mm/{};
{\ar "I";"F"}?*!/^3mm/{\cong};
\end{xy}\end{equation}

In the top row, the map $i':M\rightarrow S^n$ is the collapse map induced by the inclusion of a disk $i:D^n\hookrightarrow M$. The vertical maps are induced by maps 
\[\TopO\rightarrow\TopO_{(p)}\to B\times C\onto B.\]
where the second arrow is a homotopy inverse to the equivalence of Theorem \ref{thm:local-splitting}. 

{\it Claim.} $[M,B]\cong H^n(M;\pi_n(B))\cong bP_{n+1}\otimes\Z_{(p)}$ and similarly for $\ov M$. 

{\it Proof of Claim.} We prove this using the Atiyah--Hirzebruch spectral sequence. As discussed in \S\ref{sec:smoothing}, this spectral sequence has $E_2$-page
\[E_2^{i,-j}=H^i(M;\pi_j(B)),\]
and converges to $\mathbb E^{i-j}(M)$, where $\mathbb E^*$ denotes the cohomology theory associated to the infinite loop space $B$. In particular, to determine, $[M,B]=\mathbb E^0(M)$, we focus on the terms $E_2^{i,-i}=H^i(M;\pi_i(B))$. By our choice of $p$, and from the fact that $M$ is a closed, oriented $n$-manifold, we have 
\[H^i(M;\pi_i(B))=\begin{cases}\Z/p^a\Z&i=n\\0&\text{else}\end{cases}\]
Furthermore, by construction $\pi_k(B)=bP_{n+1}\otimes\Z_{(p)}$ is 0 for $k<n$, so the term $E_2^{n,-n}$ receives no nontrivial differentials, and the claim follows. 

The induced map
\[H^n(\pi):H^n(\ov M;\pi_n(B))\cong\Z/p^a\Z\rightarrow\Z/p^a\Z\cong H^n(M;\pi_n(B))\] is multiplication by $\deg(M\xrightarrow{\pi}\ov M)=|G|$, so $H^n(\pi)$ is the zero map because $p^a$ divides $|G|$ by construction. This explains the arrow labeled ``0" in Diagram (\ref{eqn:hyperbolic-diagram}). 

Now we conclude. On the one hand, the image of $M\#\Sigma$ under $[M,\TopO]\rightarrow [M,B]$ is nonzero because $M\#\Sigma=(i')^*(\Sigma)$, and the ``right side" of Diagram (\ref{eqn:hyperbolic-diagram}) commutes. On the other hand, $M\#\Sigma$ is in the kernel of $[M,\TopO]\rightarrow[M,B]$ because $M\#\Sigma=\pi^*(x)$, and the ``left side" of Diagram (\ref{eqn:hyperbolic-diagram}) commutes. This contradiction implies that our finite order element $g\in\Diff(N)$ must have been trivial, and this completes the proof of Theorem \ref{thm:asymmetric}. 
\end{proof}

\bibliographystyle{amsalpha}
\bibliography{refs}

Mauricio Bustamante\\
Departamento de Matem\'aticas, Universidad Cat\'olica de Chile\\
\texttt{mauricio.bustamante@mat.uc.cl}

Bena Tshishiku\\
Department of Mathematics, Brown University\\ 
\texttt{bena\_tshishiku@brown.edu}

\end{document}